\documentclass[11pt]{amsart}
\usepackage{amsmath, amsthm, amscd, amsfonts}
\allowdisplaybreaks 
\usepackage{tikz}
\usepackage{pstricks,pst-node}
\usepackage[all]{xy}

\makeatletter \oddsidemargin.9375in \evensidemargin \oddsidemargin
\newtheorem{theorem}{Theorem}[section]
\newtheorem{lemma}[theorem]{Lemma}
\newtheorem{proposition}[theorem]{Proposition}
\newtheorem{corollary}[theorem]{Corollary}
\theoremstyle{definition}
 \newtheorem{definition}[theorem]{Definition}
\newtheorem{example}[theorem]{Example}

\theoremstyle{remark}

\numberwithin{equation}{section}

\begin{document}

\title[Translation Preserving Operators on
  LCA Groups]{Translation Preserving Operators on
  Locally Compact Abelian Groups}

\author[M. Mortazavizadeh, R. Raisi Tousi, R. A. Kamyabi Gol ]{M. Mortazavizadeh, R. Raisi Tousi$^{*}$, R. A. Kamyabi Gol \\
January 30, 2019
}



\begin{abstract}
We study translation preserving operators, that is operators commuting with translations by a closed subgroup of a locally compact abelian group. We show that there is a one to one correspondence between these operators and range operators. Furthermore, we obtain a necessary condition for a translation preserving operator to be Hilbert Schmidt or of finite trace  in terms of its range operator.

\end{abstract} \maketitle

\section{Introduction and Preliminaries}

\noindent
 For a locally compact abelian (LCA) group  $G$, a translation invariant space is defined to be a closed subspace of $L^2(G)$  that is invariant under translations by elements of a closed subgroup  $\Gamma$ of $G$. Translation invariant spaces in the case  $\Gamma$ is closed, discrete and cocompact, called shift invariant spaces, have been studied in \cite{ Cab, BDR, Hel, KRr, KRs, RS}, and extended to the case of $\Gamma$ closed and cocompact (but not necessarily discrete) in \cite{BR}, see also \cite{HelS}. Recently, translation invariant spaces have been generalized in \cite{BHP} to the case when $\Gamma$  is closed (but not necessarily discrete or cocompact). In \cite{B} Bownik defined shift preserving operators as bounded linear operators on $L^2(\mathbb{R}^n)$ that commute with integer translations and characterized them via range operators. In \cite{KRsh}, the authors introduced shift preserving operators on LCA groups, whereas our goal in this paper is to define and study translation preserving operators on LCA groups. We define a translation preserving operator as a bounded linear operator on $L^2(G)$ that commutes with translations by elements of a closed subgroup $\Gamma$ of $G$ which is not necessarily discrete or cocompact. We characterize such operators in terms of range operators and define a multiplication preserving operator on a certain vector valued space. Following an idea of \cite{BHP}, we extend the results in \cite{KRsh} to the general setting when $G$ is an LCA group and $\Gamma$ is a closed subgroup of $G$. Indeed, we show that there is a one to one correspondence between translation preserving operators on $L^2(G)$ and multiplication preserving operators on the vector valued space. We use this correspondence to get the characterization of translation preserving operators in terms of range operators. We also show that a translation preserving operator has several properties in common with its associated range operator, especially compactness of one implies compactness of the other. We obtain a necessary condition for a translation preserving operator to be  Hilbert Schmidt and of finite  trace.

This paper is organized as follows. In the rest of this section, we state some required
preliminaries and notation related to multiplicatively invariant and translation invariant spaces which are studied by Bownik and Ross in \cite{BR}.
  In Section 2, we define translation preserving operators on $L^2(G)$, and give a characterization of them in terms of range operators. We achieve our goal by transforming $L^2(G)$ into a vector valued space, in such a way that translation preserving operators transfer to multiplication preserving operators. Indeed, to characterize a given translation preserving operator $U$ on $L^2(G)$, we transfer $U$  to a multiplication preserving operator $U^{'}$ on a vector valued space, and then we give a characterization of $U^{'}$ in terms of range operators. In Section 3, we find relations between some properties of translation preserving operators and the corresponding range operators. For a translation preserving operator $U$, we show that if $U$ is Hilbert Schmidt (of finite trace), then so is the range operator associated to $U$. 

Let $(\Omega, m)$ be a $\sigma$- finite measure space and $\mathcal{H}$ be a separable Hilbert space. A range function is a mapping $J: \Omega \longrightarrow \lbrace \textrm{  closed subspaces  of   $\mathcal{H}$ } \rbrace$. We write $P_{J}(\omega)$ for the orthogonal projections of $\mathcal{H}$ onto $J(\omega)$. A range function $J$ is measurable if the mapping $\omega \mapsto \langle P_{J}(\omega)(a) , b \rangle$ is measurable for all $a ,b \in \mathcal{H} $. Consider the space $L^{2}(\Omega, \mathcal{H})$ of all measurable functions $\phi$ from $\Omega$ to $\mathcal{H}$ such that $\Vert \phi \Vert _{2}^{2} = \int_{\Omega} \Vert \phi(\omega)\Vert^{2}_{\mathcal {H}} dm(\omega) <\infty$ with the inner product $\langle \phi , \psi \rangle = \int_{\Omega} \langle \phi(\omega) , \psi(\omega) \rangle_{\mathcal{H}} d m(\omega)$. It can be shown that $L^{2}(\Omega, \mathcal{H})$ is isometrically isomorphic to  $L^{2}(\Omega) \otimes \mathcal{H}$, where $\otimes$ denotes the tensor product of Hilbert spaces. A subset $\mathcal{D}$ of $L^{\infty}(\Omega) $ is said to be a determinig set for $L^{1}(\Omega) $, if for any $ f\in L^{1}(\Omega)$, $\int_{\Omega} fg dm =0 $ for all $ g \in \mathcal{D}$ implies that $f=0$. A closed subspace $W$ of $L^{2}(\Omega, \mathcal{H})$ is called multiplicatively invariant with respect to a determining set $\mathcal{D}$, if for each $\phi \in W$ and $g \in \mathcal{D}$, one has $g\phi \in W$. Bownik and Ross in \cite[Theorem 2.4]{BR}, showed that there is a correspondence between  multiplicatively invariant spaces and measurable range functions as follows.
\begin{proposition} \label{p1.1}
Suppose that $L^{2}(\Omega)$ is separable, so that $L^{2}(\Omega, \mathcal{H})$ is also separable. Then for a closed subspace $W$ of $L^{2}(\Omega, \mathcal{H})$ and a determining set $\mathcal{D}$  for $L^{1}(\Omega) $ the following are equivalent.\\
(1)  $W$ is multiplicatively invariant with respect to $\mathcal{D}$.\\
(2)  $W$ is multiplicatively invariant with respect to $L^{\infty}(\Omega)$.\\
(3) There exists a measurable range function $J$ such that 
\begin{equation*}
W=\lbrace \phi \in L^{2}(\Omega , \mathcal{H}) : \phi(\omega) \in J(\omega) \ \text{,} \ \  \text{ a.e. }  \omega \in \Omega \rbrace .
\end{equation*}
Identifying range functions which are equivalent almost everywhere, the correspondence between $\mathcal{D}$- multiplicatively invariant spaces and measurable range functions is one to one and onto. Moreover, there is a countable subset $\mathcal{A}$ of  $L^{2}(\Omega, \mathcal{H})$ such that $W$ is the smallest closed $\mathcal{D}$- multiplicatively invariant space containing $\mathcal{A}$. For any such $\mathcal{A}$ the measurable range function associated to $W$ satisfies
\begin{equation*}
J(\omega) = \overline{span} \lbrace \phi (\omega) : \phi \in \mathcal{A}\rbrace \ \  a.e. \  \omega \in \Omega .  \label{J} 
\end{equation*}
\end{proposition}

Through this paper, we assume that $G$ is a second countable LCA group and $\Gamma$ is a closed subgroup of $G$. Assume that  $\Gamma^{*}$ is the annihilator of $\Gamma$ in $\widehat{G}$. Also suppose that $\Omega$ is a measurable section for the quotient $\widehat{G} / \Gamma ^{*}$ and $C$ is a measurable section for the quotient $G / \Gamma$. For $\gamma \in \Gamma$ we denote by $X_{\gamma}$ the associated character on $\widehat{G}$, i.e. $X_{\gamma}(\chi )= \chi(\gamma)$ for  $\chi \in \widehat{G}$. One can see that the set $ \mathcal{D}= \lbrace X_{\gamma} \vert_{\Omega} : \gamma \in \Gamma \rbrace $ is a determining set for $L^{1}(\Omega) $. A closed subspace $ V \subseteq L^{2}(G)$ is called $\Gamma$- translation invariant space, if $T_{\gamma} V \subseteq V $ for all $\gamma \in \Gamma$. We say that $V$ is generated by a countable subset $\mathcal{A}$ of  $L^{2}(G)$, when $V=S^{\Gamma}(\mathcal{A})=\overline{span}\lbrace T_{\gamma}f : f \in \mathcal{A} , \gamma \in \Gamma \rbrace$.
In \cite[Proposition 6.4]{BHP} it is shown  that there exists an isometric isomorphism between $L^{2} (G)$ and $ L^{2}(\Omega , L^{2}(C))$, namely $Z: L^{2} (G) \longrightarrow L^{2}(\Omega , L^{2}(C))$  satisfying 
 \begin{equation} \label{Zak}
 Z(T_{\gamma}\phi )= X_{\gamma} \vert_{\Omega} Z (\phi).
 \end{equation}
 Let $Z$ be as in \eqref{Zak}. The forthcoming proposition, which is \cite[Theorem 6.5]{BHP}, states that $Z$ turns $\Gamma$- translation invariant  spaces in $L^{2} (G)$ into multiplicatively invariant spaces in $L^{2}(\Omega , L^{2}(C))$ with respect to the determining set  $ \mathcal{D}= \lbrace X_{\gamma} \vert_{\Omega} : \gamma \in \Gamma \rbrace$ and vice versa. It also establishes a characterization of $\Gamma$- translation invariant  spaces in terms of range functions.
\begin{proposition} \label{p1.2}
Let $V \subseteq L^{2}(G)$ be a closed subspace and $Z$ be as above. Then the following are equivalent.\\
(1)  $V$ is a $\Gamma$- translation invariant  space. \\
(2)  $Z(V)$ is a multiplicavely invariant subspace of $ L^{2}(\Omega , L^{2}(C))$ with respect to the determining set $ \mathcal{D}= \lbrace X_{\gamma} \vert_{\Omega} : \gamma \in \Gamma \rbrace $. \\
(3) There exists a measurable range function $J: \Omega \longrightarrow \lbrace  closed \ subspaces \  of \   L^{2}(C) \rbrace $ such that 
\begin{equation*}
V=\lbrace f \in L^{2}(G) : Z(f)(\omega)\in J(\omega) \text{,} \ \  \text{for a.e. }  \omega \in \Omega \rbrace.
\end{equation*}
Identifying range functions which are equivalent almost everywhere, the correspondence between $\Gamma$- translation invariant spaces and measurable range functions is one to one and onto. Moreover if $V=S^{\Gamma}(\mathcal{A})$ for some countable subset $\mathcal{A}$ of $L^{2}(G)$, the measurable range function $J$  associated to $V$ is given by
\begin{equation*}
J(\omega) = \overline{span} \lbrace Z(\phi) (\omega) : \phi \in \mathcal{A}\rbrace \ \  a.e.  \  \omega \in \Omega .  
\end{equation*}
\end{proposition}
Similar to  \cite[Theorem 5.3]{BR}, one can see the following proposition.
\begin{proposition} \label{t2.3}
Let $V$ be a $\Gamma$- translation invariant subspace of $L^{2}(G)$. Then there exist functions $\phi_{n} \in V$, $n \in \mathbb{N}$ such that,\\
(1) The set $\{ T_{\gamma} \phi _{n} \ : \ \gamma \in \Gamma \ \}$ is a continuous Parseval frame for $S^{\Gamma}(\phi _{n})$. \\
(2) The space $V$ can be decomposed as an orthogonal sum 
\begin{equation*}
V= \bigoplus _{n \in \mathbb{N}} S^{\Gamma}(\phi _{n}).
\end{equation*}
\end{proposition}

\section{A Characterization of Translation Preserving Operators}
Let the notation be as in  Section 1. A bounded linear operator $U$ on $L^2(G)$ is said to be translation preserving associated to $\Gamma$, if 
$UT_{\gamma} = T_{\gamma}U$
 where $T_{\gamma}$ is the translation operator. The main result in this section is a characterization of $\Gamma$- translation preserving operators in terms of range operators. In fact let $V$ be a $\Gamma$- translation invariant space with a range function $J$. A range operator on $J$ is a mapping $R$ from the Borel section $\Omega$ of $\widehat{G}/ \Gamma ^* $ to the set of all bounded linear operators on closed subspaces of $L^{2}(C)$, where $C$ is a Borel section of $G/ \Gamma$, so that the domain of $R(\omega)$ equals $J(\omega)$ for a.e. $\omega \in \Omega$. A range operator $R$ is called measurable, if the mapping $\omega \mapsto \langle R(\omega)P_{J}(\omega)(a) , b \rangle$ is measurable for all $a,b \in L^2(C)$, where $P_{J}(\omega)$ is the orthogonal projection of $L^{2}(C)$ onto $J(\omega)$.
\begin{definition}
Let $\mathcal{D}$ be a determining set for $L^{1}(\Omega) $  and $W \subseteq L^2(\Omega ,  \mathcal{H})$ be a $\mathcal{D}$- multiplicatively invariant space. A bounded linear operator $U: W \longrightarrow  L^2(\Omega ,  \mathcal{H})$ is said to be  multiplication preserving with respect to $\mathcal{D}$, if for all $g \in \mathcal{D}$ and $\phi \in W$ 
\begin{equation*}
U(g \phi)(\omega) = g(\omega) U(\phi)(\omega).
\end{equation*}  
\end{definition}

Now we establish a characterization of multiplication preserving  operators in terms of range operators in the following theorem. To prove this characterization theorem, we require some preliminaries related to von Neumann algebras. Let $M$ be a subset of $\mathcal{B}(\mathcal{H})$, the space of all bounded linear operators on $\mathcal{H}$. The commutant of $M$ is defined as
\begin{equation*}
M^{'}= \lbrace T \in \mathcal{B}(\mathcal{H}) : TS=ST \ \ \forall S \in M\rbrace.
\end{equation*}
The bicommutant of $M$ is defind as $M^{''} = (M^{'})^{'}$. A $*$- subalgebra $\mathcal{R}$ of $\mathcal{B}(\mathcal{H})$ is said to be a von Neumann algebra, if $\mathcal{R} = \mathcal{R}^{''}$. For comprehensive details on von Neumann algebra theory we refer to \cite{D}. Let $L^1 (\mathcal{H})$ be the space of all trace class operators on $\mathcal{H}$. By \cite[Theorem 4.2.3 ]{M}, the predual of the von Neumann algebra $\mathcal{B}(\mathcal{H})$ is isometrically isomorphic to $L^1 (\mathcal{H})$. 
%
\begin{theorem} \label{t3.3}
Suppose that $W \subseteq L^2(\Omega ,  \mathcal{H})$ is a multiplicatively invariant space with respect to a determining set $\mathcal{D}$ and $J$ is the corresponding range function. Assume that $U: W \longrightarrow  L^2(\Omega ,  \mathcal{H})$ is a bounded linear operator. Then the following are equivalent.\\
(1)  $U$ is  multiplication preserving  with respect to $\mathcal{D}$.\\
(2)  $U$ is  multiplication preserving  with respect to $L^{\infty}(\Omega) $.\\
(3) There exists a measurable range operator $R$ on $J$ such that for all $\phi \in W$,
\begin{equation} \label{r.op}
U\phi(\omega) = R(\omega)(\phi(\omega)) \ \ a.e. \ \omega \in \Omega .
\end{equation} 
Conversely, given a measurable range operator $R$ on $J$ with $ess \sup_{\omega \in \Omega} \Vert R(\omega) \Vert < \infty$,  there is a bounded multiplication preserving operator $ U : W \longrightarrow L^2(\Omega ,  \mathcal{H})$, such that \eqref{r.op} holds. The correspondence between $U$ and $R$ is one to one under the usual convention that the range operators are identified if they are equal a.e. Moreover $\Vert U \Vert = ess  \sup_{\omega \in \Omega} \Vert R(\omega) \Vert $.
\end{theorem}
\begin{proof}
First assume that (3) holds. For $\phi \in W$ and $ g \in L^{\infty}(\Omega)$ we have 
\begin{align*}
   U(g \phi )(\omega) &=  R(\omega)(g(\omega) \phi(\omega)) \cr
   &= g(\omega) R(\omega)( \phi(\omega)) \cr
   &= g(\omega) U \phi (\omega).
 \end {align*}
 Thus (2) holds. Also (2) implies (1) trivially. 
 
 Suppose that (1) holds. For $g \in L^{\infty}(\Omega)$ let $M_g$ be the multiplication operator $ M_g : L^2(\Omega, \mathcal{H}) \longrightarrow L^2(\Omega, \mathcal{H})$, $ M_{g} \phi (\omega) = g(\omega) \phi(\omega)$. Using the embedding $g\mapsto M_g$, we can consider $\mathcal{D}$ and $ L^{\infty}(\Omega)$ as subsets of $\mathcal{B}(L^2(\Omega, \mathcal{H}))$, the set of all bounded linear operators on $L^2(\Omega, \mathcal{H})$. In the context of von Neumann algebras since (1) holds, we have $\mathcal{D} \subseteq \mathcal{R}(U)^{'}$, where $\mathcal{R}(U)$ is the von Neumann algebra generated by $U$. We show that  $L^{\infty}(\Omega) \subseteq \mathcal{R}(U)^{'}$. It is enough to show that $\mathcal{D}^{'} \subseteq L^{\infty}(\Omega) ^{'}$. Suppose by contradiction $ x \in \mathcal{D}^{'} \setminus L^{\infty}(\Omega) ^{'} $. Let $y_0 \in L^{\infty}(\Omega) $ be such that $xy_0 - y_{0}x \neq 0$. Now by \cite[Theorem 4.2.3]{M}, there exists an operator  $\omega_{\xi , \eta} \in L^1 (L^2(\Omega, \mathcal{H}))$ which is continuous in the weak operator topology  such that 
 \begin{equation} \label{T}
\omega_{\xi , \eta}(xy_0 - y_{0}x) \neq 0,
\end{equation}
where $\xi , \eta \in \mathcal{H}$ and  for $T \in \mathcal{B}(L^2(\Omega, \mathcal{H}))$, 
 \begin{equation*} 
\omega_{\xi , \eta}(T) = \langle T \xi , \eta \rangle.
\end{equation*}

 We show that if $\lbrace f_{\alpha} \rbrace$ is a net in $L^{\infty}(\Omega)$ converges to $f$ in the weak-* topology, then $M_{f_{\alpha}} \rightarrow M_{f}$ in the weak operator topology. Using the fact that $L^2(\Omega, \mathcal{H}) \cong L^2 (\Omega) \otimes \mathcal{H}$,  we have for $\varphi_0 , \varphi _1 \in L^2 (\Omega) $ and $\xi _ 0 , \xi _1 \in \mathcal{H}$ 
\begin{align*}
\langle M_{f_{\alpha}} (\varphi_0 \otimes \xi_0 ) , \varphi_1 \otimes \xi_1 \rangle &= \langle f_{\alpha}(\varphi_0 \otimes \xi_0 ), \varphi_1 \otimes \xi_1 \rangle \cr
&= \langle f_{\alpha}\varphi_0 , \varphi_1 \rangle \langle \xi_0 , \xi_1  \rangle \cr
&= \langle \xi_0 , \xi_1  \rangle \int_{\Omega} f_{\alpha}(\omega) \varphi_0 (\omega) \varphi_1 (\omega)d\omega .
\end{align*}
Now since $\varphi_0 \varphi_1 \in L^1 (\Omega)$ 
\begin{equation*}
\int_{\Omega} f_{\alpha}(\omega) \varphi_0 (\omega) \varphi_1 (\omega)d\omega  \rightarrow \int_{\Omega} f(\omega) \varphi_0 (\omega) \varphi_1 (\omega)d\omega .
\end{equation*}
Thus
\begin{equation*}
\langle M_{f_{\alpha}} (\varphi_0 \otimes \xi_0 ) , \varphi_1 \otimes \xi_1 \rangle \rightarrow \langle M_{f} (\varphi_0 \otimes \xi_0 ) , \varphi_1 \otimes \xi_1 \rangle .
\end{equation*}
Now define $ F : L^{\infty}(\Omega) \longrightarrow \mathbb{C}$ by
\begin{equation*}
F(f) = \omega_{\xi , \eta}( x M_f - M_f x).
\end{equation*}
By weak operator continuity of $\omega_{\xi , \eta}$ it follows that $F$ is weak-* continuous and thus $F \in L^1 (\Omega)$. Since $x \in \mathcal{D}^{'}$, we have $F\mid_{\mathcal{D}} =0 $ and hence $F=0$ which is a contradiction to \eqref{T}. This proves (2). 

Now let (2) holds. Assume that $\mathcal{A}$ is a countable subset of $L^2(\Omega ,  \mathcal{H})$ which generates $W$. By Proposition \ref{p1.1} 
\begin{equation*}
 J(\omega) = \overline{span} \lbrace \phi (\omega) : \phi \in \mathcal{A}\rbrace \ \  a.e. \  \omega \in \Omega .
\end{equation*}
We define the operator $S(\omega) $ on the set $\lbrace \phi (\omega) :  \phi \in \mathcal{A}\rbrace$  by
\begin{equation} \label{r.o}
S(\omega) (\phi (\omega)) = U \phi (\omega).
\end{equation}
Then $S(\omega)$ is clearly linear. Also  for any $g \in L^{\infty}(\Omega)$ we have
\begin{align*}
\int_{\Omega} \vert g(\omega) \vert ^2 \Vert U \phi(\omega) \Vert ^2 d\omega &= \Vert g U \phi  \Vert_2 ^ 2 \\
&= \Vert  U g\phi  \Vert_2 ^ 2 \\
& \leq \Vert U \Vert ^2 \Vert  g\phi  \Vert _2 ^ 2 \\
&=  \Vert U \Vert ^2 \int_{\Omega} \vert g(\omega) \vert ^2 \Vert  \phi(\omega) \Vert ^2 d\omega. 
\end{align*}
Since $g \in L^{\infty}(\Omega)$ is arbitrary, this implies that 
\begin{equation} \label{ufi}
\Vert U \phi(\omega) \Vert \leq \Vert U \Vert \Vert \phi(\omega) \Vert, a.e.\  \omega \in \Omega,
\end{equation}
which shows that $S(\omega)$ is bounded and hence it is extended to a bounded linear operator $R(\omega)$ on $  \overline{span} \lbrace \phi (\omega) : \phi \in \mathcal{A} \rbrace = J(\omega) $, as is desired in (3). For the moreover part, \eqref{ufi} clearly implies that $ess\sup_{\omega \in \Omega} \Vert R(\omega) \Vert \leq \Vert U \Vert$. Also we have
\begin{align*}
\Vert U \phi \Vert _2 ^2  &= \int _{\Omega} \Vert U \phi (\omega) \Vert  ^2 d\omega \\
&=  \int _{\Omega} \Vert R(\omega) \phi (\omega) \Vert  ^2 d\omega \\
& \leq \int _{\Omega} \Vert R(\omega) \Vert ^2 \Vert \phi (\omega) \Vert  ^2 d\omega \\
& \leq ess\sup_{\omega \in \Omega} \Vert R(\omega) \Vert ^2   \int _{\Omega}  \Vert \phi (\omega) \Vert  ^2 d\omega \\
&= ess\sup_{\omega \in \Omega} \Vert R(\omega) \Vert ^2 \Vert \phi \Vert _2^2 .
\end{align*}
Thus  $\Vert U \Vert = ess\sup_{\omega \in \Omega} \Vert R(\omega) \Vert $. Finally, by \eqref{r.op} the correspondence between $U$ and $R$ is one to one and onto.
\end{proof}
Let us return to the LCA group setting. For a given LCA group $G$ with a closed subgroup $\Gamma$, let $\Omega$ be a Borel section for $\widehat{G} / \Gamma ^{*}$ and $C$ be a Borel section for $G / \Gamma$.  Now suppose that $V\subseteq L^2(G)$ is a $\Gamma$- translation invariant subspace and $U: V \longrightarrow  L^2(G)$ is a $\Gamma$- translation preserving operator. We define an induced functorial operator on the multiplicatively invariant space $Z(V)$ as
\begin{equation} \label{io}
U': Z(V) \longrightarrow  L^2(\Omega ,  L^2(C))  , \ U'(Zf)= Z(Uf),
\end{equation}
where $Z$ is defined as in Section 1, preceeding Proposition \ref{p1.2}. The following diagram explains \eqref{io}.

\[ 
\begin{psmatrix}
V & L^2(G) &\\
Z(V) & L^{2}(\Omega, L^2(C))
\psset{arrows=->, nodesep=3pt}
\ncline{1,1}{2,1}\trput{Z}
\ncline{1,2}{2,2}\trput{Z}
\ncline{1,1}{1,2}\taput{U}
\ncline{2,1}{2,2}\taput{U'}
\psset{arrows=-, nodesep=3pt}
\ncline[offset=3pt]{1,3}{2,3}
\ncline[offset=-3pt]{1,3}{2,3}
\end{psmatrix}
\]
\\
 In the sequal (Theorem \ref{t3.6}), we apply our characterization theorem (Theorem \ref{t3.3}) to the operator $U'$ so that we can  characterize the translation preserving operator $U$. First, we show that the operator $U'$ defined as \eqref{io} is multiplication preserving.
\begin{proposition} \label{p3.4}
Let $U: V \longrightarrow  L^2(G)$ be a $\Gamma$- translation preserving operator and $U'$ be the induced operator on $Z(V)$ as in \eqref{io}. Then for $f \in V$,
\begin{equation*}
U'(X_{\gamma}  Zf)(\omega)=X_{\gamma}(\omega)U'(Zf)(\omega),
\end{equation*} 
where $X_{\gamma}$ is the corresponding character on $\widehat{G}$, for $\gamma \in \Gamma$.
\end{proposition}
\begin{proof}
Using the assumption and the property of $Z$ as in \eqref{Zak}, we have
\begin{align*}
  U'(X_{\gamma}  Zf)(\omega) &= U'(Z T_{\gamma}f)(\omega) \cr
  &= Z(U T_{\gamma}f)(\omega) \cr
  &= Z(T_{\gamma} Uf)(\omega) \cr
  &= X_{\gamma}(\omega) Z(Uf)(\omega) \cr
  &= X_{\gamma}(\omega)U'(Zf)(\omega).
 \end {align*}
\end{proof}
By Theorem \ref{t3.3} and Proposition \ref{p3.4} we obtain the following corollary.
\begin{corollary} \label{cor3.5}
Assume that $V$ is a  $\Gamma$- translation invariant space and $U: V \longrightarrow  L^2(G)$ is a $\Gamma$- translation preserving operator. Then the induced operator defined by \eqref{io}  is an $L^{\infty}(\Omega) $- multiplication preserving operator.
\end{corollary}
In \cite{KRsh}, the authors characterized shift preserving operators, operators that are translation preserving with respect to a closed subgroup which is also cocompact and discrete. In the forthcoming theorem, which is the main result of this paper, we characterize translation preserving operators in the case when $\Gamma$ is not necessarily cocompact or discrete.
\begin{theorem} \label{t3.6}
Let $ V \subseteq L^{2}(G)$ be a $\Gamma$- translation invatiant subspace with range function $J$ and $U: V \longrightarrow L^2(G)$ be a bounded linear operator. Then the following are equivalent. \\
(1) The operator $U$ is translation preserving with respect to $\Gamma$. \\
(2) The induced operator $U'$ is a multiplication preserving operator with respect to $L^{\infty}(\Omega) $.\\
(3) There exists a measurable range operator $R$ on $J$ such that for all $\phi \in V$,
\begin{equation} \label{r.opg}
Z U\phi(\omega) = R(\omega)(Z\phi(\omega)) \ \ a.e. \ \omega \in \Omega .
\end{equation} 
Conversely, given a measurable range operator $R$ on $J$ with $ess \sup_{\omega \in \Omega} \Vert R(\omega) \Vert < \infty$,  there is a bounded translation preserving operator $ U : V \longrightarrow L^2(G)$, such that \eqref{r.opg} holds. The correspondence between $U$ and $R$ is one to one under the usual convention that the range operators are identified if they are equal a.e. Moreover $\Vert U \Vert = ess\sup_{\omega \in \Omega} \Vert R(\omega) \Vert $.
\end{theorem}
\begin{proof}
The implication $(1) \Rightarrow (2)$ is obvious by Corollary \ref{cor3.5}. If (2) holds, then by Theorem \ref{t3.3} there exists a measurable range operator $R$ such that for all $\phi \in W$, 
\begin{equation*}
U' Z \phi(\omega) = R(\omega)(Z\phi(\omega)) \ \ a.e. \ \omega \in \Omega ,
\end{equation*} 
hence 
\begin{equation*}
Z U\phi(\omega) = R(\omega)(Z\phi(\omega)) \ \ a.e. \ \omega \in \Omega ,
\end{equation*}
as desired in (3). Suppose that (3) holds. We have for $\gamma \in \Gamma$ and $\phi \in V$,
\begin{align*}
Z(U T_{\gamma} \phi)(\omega) &= R(\omega)(Z(T_{\gamma} \phi)(\omega)) \cr
&= R(\omega)(X_{\gamma}(\omega) Z\phi(\omega)) \cr
&= X_{\gamma}(\omega) R(\omega)(Z\phi(\omega)) \cr
&= X_{\gamma}(\omega) Z U(\phi)(\omega) \cr
&= Z( T_{\gamma} U \phi)(\omega).
\end {align*}
Now (1) follows from the fact that $Z$ is one to one. Theorem \ref{t3.3} implies that the correspondence between $R$ and $U$ is unique. The moreover part follows from Theorem \ref{t3.3} and the fact that $\Vert U \Vert = \Vert U' \Vert$.
\end{proof}

\section{Some consequences of the characterization Theorem}
In this section we establish relations between some properties of  translation preserving operators and corresponding range operators. We show that compactness of a translation preserving operator implies compactness of the associated range operator. Furthermore, using equivalent definitions of trace and Hilbert Schmidt norms, we obtain a necessary condition for a compact translation preserving operator to be Hilbert Schmidt or of finite trace. Some required  preliminaries
related to operators are given in the sequel. For more details, we refer to usual text books on operator theory, e.g.\cite{M, Zhu}. Recall that an operator $T$ on a Hilbert space $\mathcal{H}$ is called compact, if $T(B)$ is relatively compact, where $B$ is the closed unit ball in $\mathcal{H}$.
 Note that every bounded linear operator with finite rank is compact.

The following proposition states that  compactness of a translation preserving operator implies compactness of the corresponding range operator. The proof is similar to \cite[Theorem 3.1]{KRsh} in the case of $\Gamma$ discrete and cocompact, so is omitted. We remark that in \cite[Theorem 3.1]{KRsh} the additional assumptions, that $\Gamma$ is discrete and cocompact, were not used in the proof.
\begin{proposition}
Let $ V \subseteq L^{2}(G)$ be a $\Gamma$-translation invariant space with associated range function $J$. Suppose that $U: V \longrightarrow L^2(G)$ is a $\Gamma$- translation preserving operator with range operator $R$. If $U$ is compct, then so is $R(\omega)$ for almost all $\omega \in \Omega$.
\end{proposition}

Let $U$ be an operator on a Hilbert space $\mathcal{H}$, and suppose that $E$ is an orthonormal basis for $\mathcal{H}$. The Hilbert-Schmidt norm of $U$, denoted by $\Vert U \Vert _{HS}$, is defined as 
\begin{equation*}
\Vert U \Vert _{HS} = \left( \sum_{x\in E} \Vert Ux \Vert^2 \right) ^{\frac{1}{2}}.
\end{equation*}
 The operator  $U$ is called of finite trace if $tr(U) < \infty$, where
 \begin{equation*}
tr(U) = \sum_{x \in E} \langle Ux,x\rangle
\end{equation*}
  is the trace of $U$. Note that the definitions are independent of the choice of orthonormal basis. The following lemma enables us to employ Parseval frames instead of  orthonormal bases  for defining the Hilbert Schmidt norm and trace of a bounded linear operator.
 \begin{lemma} \label{l3.8}
 If $U$ is an operator on a  Hilbert space $\mathcal{H}$  and $F $ is a Parseval frame for $\mathcal{H}$, then $\Vert U \Vert _{HS} = \left( \sum_{y \in F} \Vert Uy \Vert^2 \right) ^{\frac{1}{2}}$. In particular if $U$ is positive,  $tr(U) = \sum_{y \in F} \langle Uy ,y\rangle$.
 \end{lemma}
 \begin{proof}
 Let $E$ be an orthonormal basis for $\mathcal{H}$. We have
 \begin{align*}
 \sum_{y \in F} \Vert Uy \Vert^2 &= \sum_{y \in F} \sum_{x \in E} \vert \langle Uy , x \rangle \vert ^2 \cr
&= \sum_{x \in E} \sum_{y \in F} \vert \langle y , U^* x \rangle \vert ^2 \cr
&= \sum_{n} \Vert U^* x \Vert ^2  \cr
&= \Vert U^* \Vert _{HS} ^2 = \Vert U \Vert _{HS}^2 .
 \end{align*}
Moreover, if $U$ is positive
 \begin{align*}
tr(U) &= \sum_{x \in E} \langle Ux ,x\rangle \cr
&=\sum_{x \in E} \langle U^{\frac{1}{2}} x ,U^{\frac{1}{2}} x\rangle \cr
&=  \sum_{x \in E} \sum_{y \in F} \vert \langle U^{\frac{1}{2}} x, y \rangle \vert ^2 \cr
&=  \sum_{y \in F} \sum_{x \in E} \vert \langle  x,U^{\frac{1}{2}} y \rangle \vert ^2 \cr
&= \sum_{y \in F} \Vert U^{\frac{1}{2}} y \Vert ^2 \cr
&= \sum_{y \in F} \langle Uy ,y\rangle.
 \end{align*}
 \end{proof}

Using Lemma \ref{l3.8}, we have the following proposition which shows that a compact range operator is Hilbert Schmidt or of finite trace whenever the associated translation preserving operator has the same properties. It generalizes \cite[Theorem 3.3]{KRsh}.
\begin{proposition}
Suppose that $ V \subseteq L^{2}(G)$ is a $\Gamma$- translation invariant space with associated range function $J$. Let $U: V \longrightarrow V$ be a compact  $\Gamma$- translation preserving operator with range operator $R$.\\
(1) If $U$ is a Hilbert Schmidt operator then so is $R(\omega)$ for a.e. $\omega \in \Omega$. \\
(2) If $U$ is positive and of finite trace then so is $R(\omega)$ for a.e. $\omega \in \Omega$.

\end{proposition} 
\begin{proof}
First note that with the notation as in Proposition \ref{t2.3}, $\{ T_{\gamma} \phi _{n}  :  \gamma \in \Gamma , \ n \in \mathbb{N} \}$ is a continuous Parseval frame for $V= \bigoplus _{n \in \mathbb{N}} S^{\Gamma}(\phi _{n})$ and hence by \cite[Theorem 6.6]{BHP}, the set $\{ Z(\phi _{n}) (\omega) : \ n \in \mathbb{N} \} $ is a Parseval frame for $J(\omega)$, for almost every $\omega \in \Omega$. Now let $U$ be Hilbert Schmidt. By Lemma \ref{l3.8},
\begin{equation*}
\int_{\Gamma} \sum_{ n \in \mathbb{N}} \Vert T_{\gamma} U \phi _{n} \Vert ^2 dm_{\Gamma}(\gamma) = \sum_{ n \in \mathbb{N}} \int_{\Gamma} \Vert U(T_{\gamma} \phi _{n}) \Vert ^2 dm_{\Gamma}(\gamma)  = \Vert U \Vert _{HS}^2 < \infty .
\end{equation*}
Using the fact that $Z$ is isometry and Theorem \ref{t3.6}, we obtain 
\begin{align*}
\infty &>  \sum_{ n \in \mathbb{N}} \Vert U \phi _{n} \Vert ^2 \cr
&= \sum_{ n \in \mathbb{N}}\Vert Z  U \phi _{n}  \Vert ^2 \cr
&= \sum_{ n \in \mathbb{N}} \int _{\Omega} \Vert Z  U \phi _{n} (\omega) \Vert ^2 d \omega \cr
&= \int _{\Omega} \sum_{ n \in \mathbb{N}}\Vert Z U \phi _{n} (\omega) \Vert ^2 d \omega \cr
&= \int _{\Omega} \sum_{ n \in \mathbb{N}} \Vert R(\omega) (Z(\phi _{n})(\omega)) \Vert ^2 d\omega \cr
&= \int _{\Omega} \Vert R(\omega) \Vert _{HS} ^2 \  d\omega , 
\end{align*}
which shows that $R(\omega)$ is Hilbert Schmidt,  for a.e. $\omega \in \Omega$. 
If $U$ is positive and of finite trace, then by Lemma \ref{l3.8}, the fact that $Z$ is isometry (in the second equality below), and Theorem \ref{t3.6} (in the third equality below),
\begin{align*}
\infty > tr(U)  &= \sum_{ n \in \mathbb{N}} \int_{\Gamma} \langle U T_{\gamma} \phi _{n} , T_{\gamma} \phi _{n} \rangle dm_{\Gamma}(\gamma) \cr
&= \sum_{ n \in \mathbb{N}} \int_{\Gamma} \langle Z U T_{\gamma} \phi _{n} , Z T_{\gamma} \phi _{n} \rangle dm_{\Gamma}(\gamma) \cr
&= \sum_{  n \in \mathbb{N}}  \int_{\Gamma} \int _{\Omega} \langle R(\omega )( Z T_{\gamma} \phi _{n} (\omega)) , Z T_{\gamma} \phi _{n} (\omega)  \rangle d \omega dm_{\Gamma}(\gamma)  \cr
&= \sum_{  n \in \mathbb{N}} \int_{\Gamma} \int _{\Omega} \langle R(\omega )( X _{\gamma} Z  \phi _{n} (\omega)) , X_{\gamma} Z \phi _{n} (\omega)  \rangle d \omega dm_{\Gamma}(\gamma) \cr
&= \int_{\Gamma} \sum_{  n \in \mathbb{N}}  \int _{\Omega} \langle R(\omega )( X _{\gamma} Z  \phi _{n} (\omega)) , X_{\gamma} Z \phi _{n} (\omega)  \rangle d \omega dm_{\Gamma}(\gamma) .
\end{align*}
So we have
\begin{align*}
\infty &> \sum_{ n \in \mathbb{N}} \int _{\Omega} \langle R(\omega )(  Z  \phi _{n} (\omega)) , Z \phi _{n} (\omega)  \rangle d \omega  \cr
&=  \int _{\Omega} \sum_{ n \in \mathbb{N}} \langle R(\omega )(  Z  \phi _{n} (\omega)) , Z \phi _{n} (\omega)  \rangle d\omega.
\end{align*}
Thus by Lemma \ref{l3.8}, $R(\omega)$ is of finite trace, for a.e. $\omega \in \Omega$. 
\end{proof} 
The next proposition states that a necessary and sufficient condition for a translation preserving operator to be isometric (self adjoint) is that its corresponding range operator is  isometric (self adjoint). The proof is similar to \cite[Propositions 3.4, 3.5]{KRsh} and is omitted.
\begin{proposition}
Suppose that $ V \subseteq L^{2}(G)$ is a $\Gamma$- translation invariant space with associated range function $J$. Let $U: V \longrightarrow V$ be a compact  $\Gamma$- translation preserving operator with range operator $R$. Then\\
(1) $U$ is isometry if and only if so is $R(\omega)$ for almost every $\omega \in \Omega$. \\
(2) $U$ is self adjoint if and only if so is $R(\omega)$ for almost every $\omega \in \Omega$.
\end{proposition}

\begin{example}
Define $U  :   L^2(\Bbb{R}) \longrightarrow  L^2(\Bbb{R}) $ by $Uf(x) = f(x) - f(x-1)$. Clearly $U$ is a translation preserving operator. By Theorem \ref{t3.6} there exists a range operator $R$ such that for every $\phi \in L^2(\Bbb{R})$,
\begin{equation*}
R(\omega) (Z\phi (\omega)) = (Z U) \phi (\omega) = Z\phi (\omega) + \exp (i\omega) Z\phi (\omega) = (1+\exp (i\omega))  Z\phi (\omega) .
\end{equation*}
Note that $R(\omega)$ is a multiplication operator (multiplication by $1+\exp (i\omega)$) which is not compact. Notice that also $U$ is not compact. Moreover, $U^* f (x) = f(x) + f(x-1)$ and $R(\omega)^* (Z\phi (\omega)) =(1- \exp (i\omega))  Z\phi (\omega) $, where $f , \phi \in L^2(\Bbb{R})$.
\end{example}

\begin{example}
For a fixed prime number $p$, the $p$-adic numbers $\mathbb{Q}_p$  is the completion of rational numbers $\mathbb{Q}$ under the $p$-adic norm $\vert  .  \vert _p$ defined as follows. Every nonzero rational $x$ can be uniquely written as  $ x= \frac{r}{s}p^n$, where $r, s, n \in \mathbb{Z}$ and $p$ does not divide $r$ or $s$. We then define the $p$-adic norm of $x$ by $ \vert x \vert _p = p^{-n}$, in addition $\vert 0 \vert _p = 0$. Then $\mathbb{Q}_p$ is an additive LCA group and 
 $\mathbb{I}_p := \lbrace x \in  \mathbb{Q}_p :  \vert x \vert _p  \leq 1  \rbrace$ is a closed, compact, and open subgroup of $\mathbb{Q}_p$. Now consider the operator $U : L^2(\mathbb{Q}_p) \longrightarrow L^2(\mathbb{Q}_p)$, given by $U\phi(x) = \overline{\phi(x)}$. Then $U$ is clearly an $\mathbb{I}_p$-translation preserving operator. By Theorem \ref{t3.6}, there exists a range operator $R$ such that for any $\phi \in L^2(\mathbb{Q}_p)$,
\begin{equation*}
R(\omega) (Z\phi (\omega)) = (Z U) \phi (\omega)=Z \overline{\phi}(\omega) = \overline{Z\phi(- \omega)}.
\end{equation*}
Note that $U$ and $R(\omega)$ are both isometries.
\end{example}



\textbf{Acknowledgements}
We are indebted to Professor Hartmut Fuhr for valuable comments and remarks on the proof of Theorem 2.2. We thank Professor Kenneth A. Ross for stimulating discussions on an earlier version of this paper. We also thank Dr. R. Fa'al for his usefull suggestions.

\bibliographystyle{amsplain}

\end{document}